\documentclass[10pt]{article}

\usepackage[english]{babel}
\usepackage{amsmath}
\usepackage{amssymb}
\usepackage{bm}
\usepackage{bbm}
\usepackage{mathrsfs,color}
\usepackage{graphics,graphicx,theorem}
\usepackage{dsfont}
\usepackage{setspace}

\usepackage{mathtools}

\usepackage{natbib}
\usepackage{multirow}
\usepackage{hhline}
\usepackage{hyperref}

\usepackage{authblk}

\hypersetup{
	colorlinks=true,
	urlcolor=blue,
	citecolor=blue,
	linktoc=all,
	linkcolor=red} 

\usepackage[scriptsize]{subfigure}
\allowdisplaybreaks

\makeatletter
\renewcommand\@biblabel[1]{}
\makeatother
\usepackage{amsmath}
\usepackage{bm}
\usepackage{amsfonts,dsfont,mathrsfs}
\newcommand{\R}{\mathds{R}}
\newcommand{\E}{\mathds{E}}

\newcommand{\N}{\mathds{N}}

\newcommand{\Z}{\mathds{Z}}
\renewcommand{\P}{\mathds{P}}
\newcommand{\W}{\mathds{W}}

\newcommand{\Bcr}{\mathscr{B}}

\newcommand{\Lcr}{\mathscr{L}}

\newcommand{\Mcr}{\mathscr{M}}

\newcommand{\Acr}{\mathscr{A}}

\newcommand{\Wcr}{\mathscr{W}}
\newcommand{\Zcr}{\mathscr{Z}}

\newcommand{\D}{{\rm d}}

\def\simiid{\stackrel{\mbox{\scriptsize{iid}}}{\sim}}
\newcommand{\indic}{\mathds{1}}

\newtheorem{Theorem}{Theorem}[section]
\newtheorem{Definition}{Definition}[section]

\newtheorem{Corollary}{Corollary}[section]
\newtheorem{Remark}{Remark}[section]

\newenvironment{proof}{\noindent \textit{Proof.}}{\hfill$\square$}

\usepackage[top=1.2in,bottom=1.2in,left=1.2in,right=1.2in]{geometry}

\begin{document}

\title{\bf {\Large{On Johnson's ``sufficientness" postulates for features-sampling models}}}

\author[,1]{Federico Camerlenghi \thanks{Also affiliated to Collegio Carlo Alberto, Piazza V. Arbarello 8, Torino, and BIDSA, Bocconi University, Milano, Italy; federico.camerlenghi@unimib.it}}
\author[,2]{Stefano Favaro \thanks{Also affiliated to Collegio Carlo Alberto, Piazza V. Arbarello 8, Torino, and IMATI-CNR ``Enrico  Magenes", Milan, Italy; stefano.favaro@unito.it}}

\affil[1]{University of Milano - Bicocca, Piazza dell'Ateneo Nuovo 1, Milano}
\affil[2]{University of Torino, Corso Unione Sovietica 218/bis, Torino}

\date{}
\maketitle
\thispagestyle{empty}

\setcounter{page}{1}

\begin{abstract}
In the 1920's, the English philosopher W.E. Johnson introduced a characterization of the symmetric Dirichlet prior distribution in terms of its predictive distribution. This is typically referred to as Johnson's ``sufficientness'' postulate, and it has been the subject of many contributions in Bayesian statistics, leading to predictive characterization for infinite-dimensional generalizations of the Dirichlet distribution, i.e. species-sampling models. In this paper, we review ``sufficientness'' postulates for species-sampling models, and then investigate analogous predictive characterizations for the more general features-sampling models. In particular, we present a ``sufficientness'' postulate for a class of features-sampling models referred to as Scaled Processes (SPs), and then discuss analogous characterizations in the general setup of  features-sampling models.
\end{abstract}


\maketitle

\section{Introduction}
Exchangeability (\citet{definetti}) provides a natural modeling assumption in a large variety of statistical problems, and it amounts to assume that the order in which observations are recorded is not relevant. Consider a sequence of random variables $(Z_{j})_{j \geq 1 }$ defined on a common probability space $(\Omega, \Acr, \P)$ and taking values in an arbitrary space, which is assumed to be Polish. The sequence $(Z_{j})_{j \geq 1 }$ is exchangeable if and only if 
\[
(Z_1, \ldots , Z_n) \stackrel{\rm d}{=} (Z_{\sigma (1)}, \ldots , Z_{\sigma (n)})
\]
for any permutation $\sigma$ of the set $\{ 1, \ldots , n \}$ and any $n \geq 1$.  By virtue  of the celebrated de Finetti representation theorem, exchangeability of $(Z_{j})_{j \geq 1 }$ is tantamount to assert the existence of a random element $\tilde{\mu}$, defined on a (parameter) space $\Theta$, such that conditionally on $\tilde{\mu}$ the $Z_{j}$s are independent and identically distributed with common distribution $p_{\tilde{\mu}}$, i.e., 
\begin{equation}
\label{eq:exchangeability}
\begin{split}
Z_j\, |\, \tilde{\mu}  & \simiid   p_{\tilde{\mu}}  \quad j \geq 1  \\
\tilde{\mu}  & \sim \Mcr,
\end{split}  
\end{equation}
where $\Mcr$ is the distribution of $\tilde{\mu}$. In a Bayesian setting, $\Mcr$ takes on the interpretation of a prior distribution for the parameter object of interest. In this sense, the de Finetti representation theorem is a natural framework for Bayesian statistics. For mathematical convenience, $\Theta$ is assumed to be a Polish space, equipped with the Borel $\sigma$-algebra $\Bcr (\Theta)$. Hereafter, with the term \textit{parameter} we refer to both a finite and an infinite dimensional object.

Within the framework of exchangeability \eqref{eq:exchangeability}, a critical role is played by the predictive distributions, namely the conditional distributions of the $(n+1)$th observation $Z_{n+1}$, given $\bm{Z}_n := (Z_1, \ldots , Z_n)$. The problem of characterizing prior distributions $ \Mcr$ in terms of their predictive distributions has a long history in Bayesian statistics, starting from the seminal work of the English philosopher \citet{johnson1932} who provided a predictive characterization of the symmetric Dirichlet prior distribution. Such a characterization is typically referred to  as Johnson's ``sufficientness'' postulate. Species-sampling models (\citet{Pitman96species}) provide arguably the most popular infinite-dimensional generalization of the Dirichlet distribution. They form a broad class of nonparametric prior models which correspond to assume that $p_{\tilde{\mu}}$ in \eqref{eq:exchangeability} is an almost surely discrete random probability measure
\begin{equation}
\label{eq:ssm_def}
\tilde{p}= \sum_{i \geq 1} \tilde{p}_i \delta_{\tilde{z}_i},
\end{equation}
where: i) $(\tilde{p}_i)_{i \geq 1}$ are non-negative random weights summing up to $1$ almost surely; ii) $(\tilde{z}_i)_{i \geq 1}$ are random species' labels, independent of $(\tilde{p}_i)_{i \geq 1}$, and i.i.d. with common (non-atomic) distribution $P$. The term \textit{species} refers to the fact that the law of $\tilde{p}$ is a prior distribution for the unknown species composition $(\tilde{p}_i)_{i \geq 1}$ of a population of individuals $Z_{j}$s, with $Z_{j}$ belonging to a species $\tilde{z}_{i}$ with probability $\tilde{p}_{i}$, for $j,i\geq1$. In the context of species-sampling models, \citet{Regazzini1978} and \citet{Lo1991}, provided a ``sufficientness'' postulate for the Dirichlet process (\citet{ferguson1973bayesian}). Such a characterization has then been extended by \citet{Zabell2005} to the Pitman-Yor process (\citet{Perman92,pityor_97}), and by \citet{marcoB} to the more general Gibbs-type prior models (\citet{gnepit_06}).

In this paper, we introduce and discuss Johnson's ``sufficientness'' postulates in the features-sampling setting, which generalizes the species-sampling setting by allowing each individual of the population to belong to multiple species, now called features. We point out that feature-sampling models are extremely important in different applied areas, see, e.g.
\citet{Gri_11,Ayed2021} and references therein.
Under the framework of exchangeability \eqref{eq:exchangeability}, the features-sampling setting assumes that 
\begin{equation} \label{eq:counting_measure}
Z_j | \tilde{\mu}   = \sum_{i \geq 1} A_{j,i} \delta_{\tilde{w}_i}  \sim p_{\tilde{\mu}},
\end{equation}
and
\begin{displaymath}
\tilde{\mu} = \sum_{i \geq 1} \tilde{p}_i \delta_{\tilde{w}_i}
\end{displaymath}
where: i) conditionally on $\tilde{\mu}$, $(A_{j,i})_{i \geq 1}$ are independent Bernoulli random variables with parameters $(\tilde{p}_i)_{i\geq 1}$; ii) $(\tilde{p}_i)_{i \geq 1}$ are $(0,1)$-valued random weights; iii) $(\tilde{w}_i)_{i \geq 1}$ are random features' labels, independent of $(\tilde{p}_i)_{i \geq 1}$, and i.i.d. with common (non-atomic) distribution $P$. That is, individual $Z_j$ displays feature $\tilde{w}_i$ if and only if $A_{j,i}=1$, which happens with probability $\tilde{p}_i$. 
For example, if, conditionally on $\tilde{\mu}$, $Z_j $ displays only two features,  say $\tilde w_1$ and $\tilde w_5$, it equals the random measure $\delta_{\tilde w_1}+\delta_{\tilde w_5}$.
The distribution $p_{\tilde{\mu}}$ is the law of a Bernoulli process with parameter $\tilde{\mu}$, which is denoted by ${\rm BeP} (\tilde{\mu})$, whereas the law of $\tilde{\mu}$ is a nonparametric prior distribution for the unknown feature probabilities $(\tilde{p}_i)_{i \geq 1}$, i.e. a features-sampling model. Here, we investigate the problem of characterizing prior distributions for $\tilde{\mu}$ in terms of their predictive distributions, with the goal to provide ``sufficientness'' postulates for features-sampling models. We discuss such a problem, and present partial results for a class of features-sampling models referred to as Scaled Process (SP) priors for $\tilde{\mu}$ (\citet{james2015scaled,camerlenghi2021scaled}). With these results, we aim at stimulating future research in this field, to obtain ``sufficientness'' postulates for general features-sampling models.

The paper is structured as follows. In Section \ref{sec:species}, we present a brief review on Johnson's ``sufficientness'' postulates for species-sampling models. Section \ref{sec:BeP} focuses on nonparametric prior models for the Bernoulli process, i.e. features-sampling models: we review their definitions, properties and sampling structures. In Section \ref{sec:SP_predictive} we present a ``sufficientness'' postulate for SPs. Section \ref{discuss} concludes the paper by discussing our results and conjecturing analogous results for more general classes of features-sampling models.

%%%%%%%%%%%%%%%%%%%%%%%%%%%%%%%%
%%%%%%%%%%%%%%%%%%%%%%%%%%%%%%%%
%%%%%%%%%%%%%%%%%%%%%%%%%%%%%%%%
%%%%%%%%%%%%%%%%%%%%%%%%%%%%%%%%

\section{Species-sampling models}  \label{sec:species}

To introduce species-sampling models, we  assume that the observations are $\Z$-valued random elements, and $\Z$ is supposed to be a Polish space whose Borel $\sigma$-algebra is denoted by $\Zcr$. Thus $\Z$ contains all the possible species' labels of the populations. When we deal with species-sampling models, the hierarchical formulation \eqref{eq:exchangeability} specializes as
\begin{equation}\label{eq:SSM}
\begin{split}
Z_j | \tilde{p}  & \simiid   \tilde{p}  \quad j \geq 1  \\
\tilde{p}  & \sim \Mcr
\end{split}  
\end{equation}
where $\tilde{p}= \sum_{i \geq 1} \tilde{p}_i \delta_{\tilde{z}_i}$ is an almost surely discrete random probability measure on $\Z$, and $\Mcr$ denotes its law. We also remind that: i) $(\tilde{p}_i)_{i \geq 1}$ are non-negative random weights summing up to $1$ almost surely; ii) $(\tilde{z}_i)_{i \geq 1}$ are random species' labels, independent of $(\tilde{p}_i)_{i \geq 1}$, and  i.i.d. as a common (non-atomic) distribution $P$. Using the terminology of \citet{Pitman96species}, the discrete random probability measure $\tilde{p}$ is a \textit{species-sampling model}. In Bayesian nonparametrics, popular examples of species-sampling models are: the Dirichlet process (\citet{ferguson1973bayesian}), the Pitman-Yor process (\citet{Perman92,pityor_97}), the normalized generalized Gamma process (\citet{Brix1999,Lijoi2007}). These are examples belonging to a peculiar subclass of species-sampling models, which are referred to as Gibbs-type prior models (\citet{gnepit_06,gibbs_15}). More general subclasses of species-sampling models are, e.g. the homogeneous normalized random measures  (\citet{regazzini2003}) and the Poisson-Kingman models (\citet{pitman2003poisson,Pitman_06}). We refer to \citet{lijoi2010models} and \citet{fundamentals} for a detailed and stimulating account on species-sampling models and their use in Bayesian nonparametrics.

Because of the almost sure discreteness of $\tilde{p}$ in \eqref{eq:SSM}, a random sample $\bm{Z}_n := (Z_1, \ldots , Z_n)$ from $\tilde{p}$ features ties, that is $\P (Z_{j_1} = Z_{j_2})>0$ if $j_1 \not = j_2$. Thus, $\bm{Z}_n$ induces a random partition of the set $\{1,\ldots,n\}$ into $K_n =k \leq n$ blocks, labelled by $Z_1^*, \ldots , Z_{K_n}^*$, with corresponding frequencies $(N_{n,1}, \ldots , N_{n,K_n})= (n_1, \ldots , n_k)$, such that $N_{i,n}\geq1$ and $\sum_{1\leq i\leq K_{n}}N_{i,n}=n$. From \citet{Pitman96species}, the predictive distribution of $\tilde{p}$ is of the form
\begin{equation}
\label{eq:SS_predictive}
\P (Z_{n+1} \in A | \bm{Z}_n) = g(n,k,\bm{n})P(A) + \sum_{i=1}^k  f_i (n,k, \bm{n}) \delta_{Z_i^*} (A) , \quad A \in \Zcr,
\end{equation}
for any $n\geq1$, having set $\bm{n}= (n_1, \ldots , n_k)$, with $g$ and $f_{i}$ being arbitrary non-negative functions that satisfy the constraint $g(n,k,\bm{n}) + \sum_{i=1}^k  f_i (n,k, \bm{n})=1$. The predictive distribution \eqref{eq:SS_predictive} admits the following interpretation: i) $g(n,k,\bm{n})$ corresponds to the probability that $Z_{n+1}$ is a new species, that is a species non observed in $\bm{Z}_n$; ii) $f_i (n,k, \bm{n})$ corresponds to the probability that $Z_{n+1}$ is a species $Z_i^*$ in $\bm{Z}_n$. The functions $g$ and $f_{i}$ completely determine the distribution of the exchangeable sequence $(Z_{j})_{j\geq1}$, and in turns the distribution of the random partition of $\mathbb{N}$ induced by $(Z_{j})_{j\geq1}$. Predictive distributions of popular species-sampling models, e.g. the Dirichlet process, the Pitman-Yor process and the normalized generalized Gamma process, are of the form  \eqref{eq:SS_predictive} for suitable specification of the functions $g$ and $f_{i}$. We refer to \citet{Pitman_06} for a detailed account on random partitions induced by species-sampling models, and generalizations thereof.

Here, we recall the predictive distribution of Gibbs-type prior models (\citet{gnepit_06,gibbs_15}). Let us first introduce the definition of these processes.
\begin{Definition}
Let $\sigma\in(-\infty,1)$ and let $P$ be a (non-atomic) distribution on $(\Z, \Zcr)$. A Gibbs-type prior models is a species-sampling models with predictive distribution of form
\begin{equation}
\begin{split}
\P (Z_{n+1} \in A | \bm{Z}_n) & = \frac{V_{n+1,k+1}}{V_{n,k}} P(A) +  \frac{V_{n+1,k}}{V_{n,k}}  \sum_{i=1}^k  (n_i-\sigma)
\delta_{Z_i^*} (A), \quad A \in \Zcr,
\end{split}
\end{equation}
for any $n \geq 1$, where $\{ V_{n,k} : \; n \geq 1, \, 1 \leq k \leq n \}$ is a collection of non-negative weights that satisfy the recurrence  relation $V_{n,k} = (n -\sigma k) V_{n+1,k} +V_{n+1, k+1}$ for all $k =1, \ldots , n$, $n \geq 1$, with the proviso $V_{1,1}=1$.
\end{Definition}
Note that the Dirichlet process is a Gibbs-type prior model which corresponds to
\[
V_{n,k} = \frac{\theta^k}{(\theta)_n}
\]
for $\theta>0$, and having denoted by $(a)_b= \Gamma (a+b)/\Gamma (a)$ the Pochhammer symbol for the rising factorials. Moreover, the Pitman-Yor process is a Gibbs-type prior model corresponding to
\[
V_{n,k} = \frac{\prod_{i=0}^{k-1} (\theta+i \sigma)}{(\theta)_n}
\]
for $\sigma \in (0,1)$ and $\theta > -\alpha$. We refer to \citet{pitman2003poisson} for other examples of Gibbs-type prior models, and for a detailed account on the $V_{n,k}$s. See also \citet{Pitman_06} and references therein.

Because of de Finetti's representation theorem, there exists a one-to-one correspondence between the functions $g$ and $f_i$ in the predictive distribution \eqref{eq:SS_predictive} and the law $\Mcr$ of $\tilde{p}$, i.e. the de Finetti measure. This is at the basis of Johnson's ``sufficientness'' postulates, characterizing species-sampling models through their predictive distributions. \citet{Regazzini1978}, and later \citet{Lo1991}, provided the first ``sufficientness'' postulate for species-sampling models, showing that: the Dirichlet process is the unique species-sampling model for which the function $g$ depends on $\bm{Z}_n$ only through $n$, and the function $f_i$ depends on $\bm{Z}_n$ only through $n$ and $n_{i}$, for $i\geq1$. Such a result has been extended in \citet{Zabell1997}, providing the following ``sufficientness'' postulate for the Pitman-Yor process: the Pitman-Yor process is the unique species-sampling model for which the function $g$ depends on $\bm{Z}_n$ only through $n$ and $k$, and the function $f_{i}$ depends on $\bm{Z}_n$ only through $n$ and $n_{i}$, for $i\geq1$. \citet{marcoB} discussed ``sufficientness'' postulate in the more general setting of Gibbs-type prior models, showing that: Gibbs-type prior models are the sole species-sampling models for which the function $g$ depends on $\bm{Z}_n$ only through $n$ and $k$, and the function $f_{i}$ depends on $\bm{Z}_n$ only through $n$, $k$ and $n_{i}$. Such a result shows a critical difference, at the sampling level, between the Pitman-Yor process and Gibbs-type prior models, which lies in the inclusion of the sampling information on the observed number of distinct species in the probability of observing at the $(n+1)$-th draw a species already observed in the sample.

%%%%%%%%%%%%%%%%%%%%%%%%%%%%%%%%
%%%%%%%%%%%%%%%%%%%%%%%%%%%%%%%%
%%%%%%%%%%%%%%%%%%%%%%%%%%%%%%%%
%%%%%%%%%%%%%%%%%%%%%%%%%%%%%%%%

\section{Features-sampling models} \label{sec:BeP}

Features-sampling models generalize species-sampling models by allowing each individual to belong to more than one species, which are now called features. To introduce features-sampling models, we consider a space of features $\W$, which is assumed to be a Polish space, and we denote by $\Wcr$ its Borel $\sigma$-field. Thus $\W$ contains all the possible features' labels of the population. Observations are represented through the counting measure \eqref{eq:counting_measure}, whose parameter $\tilde{\mu}$ is an almost surely discrete measure with masses in $(0,1)$. When we deal with features-sampling models, the hierarchical formulation \eqref{eq:exchangeability} specializes as
\begin{equation}
\label{eq:feature_ex}
\begin{split}
Z_j | \tilde{\mu}  & \simiid {\rm BeP} (\tilde{\mu}) \\
\tilde{\mu}   &  \sim  \Mcr
\end{split}
\end{equation}
where $\tilde{\mu} = \sum_{i \geq 1} \tilde{p}_i \delta_{\tilde{w}_i}$ is an almost surely discrete random measure on $\W$, and $\Mcr$ denotes its law. We also remind that: i) conditionally on $\tilde{\mu}$, $(A_{j,i})_{i \geq 1}$ are independent Bernoulli random variables with parameters $(\tilde{p}_i)_{i\geq 1}$; ii) $(\tilde{p}_i)_{i \geq 1}$ are $(0,1)$-valued random weights; iii) $(\tilde{w}_i)_{i \geq 1}$ are random features' labels, independent of $(\tilde{p}_i)_{i \geq 1}$, and i.i.d. with common (non-atomic) distribution $P$. Completely random measures (CRMs) (\citet{daleyII,kingman1967completely}) provide a popular class of nonparametric priors $\Mcr$, the most common examples being the Beta process prior and the stable Beta process prior (\citet{teh2009indian,james2017bayesian}). See also \citet{broderick2018posteriors} and references therein for other examples of CRM priors, and generalizations thereof. Recently \citet{camerlenghi2021scaled} investigated an alternative class of nonparametric priors $\Mcr$, generalizing CRM priors and referred to as Scaled Processes (SPs). SPs priors first appeared in the work of \citet{james2017bayesian}.

We assume a random sample $\bm{Z}_n :=(Z_1, \ldots , Z_n)$ to be modeled as in \eqref{eq:feature_ex}, and we introduce the predictive distribution of $\tilde{\mu}$, that is the conditional probability of $Z_{n+1}$ given $\bm{Z}_n$. Note that, because of the pure discreteness of $\tilde{\mu}$, the observations $\bm{Z}_n$ may share a random number of distinct features, say $K_n=k$, denoted here as $W_1^*, \ldots , W_{K_n}^*$, and each feature $W_i^*$ is displayed exactly by $M_{n,i}=m_i$ of the $n$ individuals, as $i = 1, \ldots , k$.  Since the features' labels are immaterial and i.i.d. form  the base measure $P$, the conditional distribution  of  $Z_{n+1}$, given $\bm{Z}_n$, may be equivalently characterized through the vector $(Y_{n+1}, A_{n+1,1}^*, \ldots , A_{n+1,K_n}^*)$, where: i) $Y_{n+1}$ is the number of new features displayed by the $(n+1)$th individual, namely hitherto unobserved out of the sample $\bm{Z}_n$; ii) $A_{n+1,i}^*$ is a $\{0,1\}$-valued random variable for any $i =1, \ldots , K_n$, and $A_{n+1,i}^*=1$ if the $(n+1)$th individual displays feature $W_i^*$, it equals $0$ otherwise. Hence, the predictive distribution of $\tilde{\mu}$ is
\begin{equation}
\label{eq:predictive_general}
\begin{split}
&\P  ((Y_{n+1}, A_{n+1,1}^*, \ldots , A_{n+1,K_n}^*) = (y, a_1, \ldots , a_{K_n}) | \bm{Z}_n)  =
f(y, a_1, \ldots , a_k ; n, k, \bm{m})
\end{split}
\end{equation}
where we denote by $f$ a probability distribution evaluated at $(y, a_1, \ldots, a_k)$, and where $n,k$ and $\bm{m}:=( m_1, \ldots , m_k)$ is the sampling information. In the rest of this section we specify the function $f$ under the assumption of a CRM prior and a SP prior, showing its dependence on $n, K_n$ and $(M_{n,1}, \ldots , M_{n, K_n})$. In particular, we show how SP priors allow to enrich the predictive distribution of CRM priors, by including additional sampling information in terms of the number of distinct features and their corresponding frequencies.

\subsection{Priors based on CRMs}

Let $\mathsf{M}_\mathds{W}$ denote the space of all bounded and finite measures on $(\W, \Wcr)$, that is to say  $\mu \in \mathsf{M}_\W$ iff $\mu (A)< +\infty$ for any bounded set $A\in \Wcr$. Here we recall the definition of a Completely Random Measure (CRM) (see, e.g., \citet{daleyII}).
\begin{Definition}
A Completely Random Measure (CRM) $\tilde{\mu}$ on $(\W, \Wcr)$ is a random element taking values in the space $\mathsf{M}_\W$  such that the random variables $\tilde{\mu} (A_1), \ldots , \tilde{\mu} (A_n)$ are independent for any choice of bounded and disjoint sets $A_1, \ldots , A_n \in \Wcr$ and for any $n \geq 1$. 
\end{Definition}
We remind that \citet{kingman1967completely} proved that a CRM may be decomposed as the sum of  a deterministic drift and  a purely atomic component. In Bayesian nonparametrics, it is common to consider purely atomic CRMs without fixed points of discontinuity, that is to say $\tilde{\mu}$ may be represented as $\tilde{\mu} := \sum_{i \geq 1} \tilde{\eta}_i \delta_{\tilde{w}_i}$, where
$(\tilde{\eta}_i)_{i \geq 1}$ is a sequence of random atoms and $(\tilde{w}_i)_{i \geq 1}$ are the random locations. An appealing property of purely atomic CRMs is the availability of their Laplace functional, indeed for any measurable function $f : \W \to \R^+$ one has
\begin{equation}
\label{eq:laplace}
\E \left[ e^{-\int_\W  f(w) \tilde{\mu} (\D w)} \right] =
\exp\left\{ - \int_{\W \times\R^+} (1-e^{-s f(w)}) \nu ( \D w , \D s)\right\} 
\end{equation}
where $\nu$ is a measure on $\W  \times \R^+$ called the L\'evy intensity of the CRM $\tilde{\mu}$ and it is such that
\begin{equation}
\label{eq:integrability}
\nu (\{w\} \times \R^+) = 0 \quad \forall w \in \W , \quad \text{and } \int_{ A  \times \R^+} \min \{s,1  \} \nu (\D w, \D s) < \infty
\end{equation}
for any bounded Borel set $A$. Here, we focus on homogeneous CRMs by assuming that the atoms $\tilde{\eta}_i$s and the locations $\tilde{w}_i$s are independent, in this case the L\'evy measure may be written as
\[
\nu (\D w , \D s ) = \lambda (s) \D s P (\D w)
\]
for some measurable function $\lambda: \R^+ \to \R^+$ and a probability measure $P$ on $(\W , \Wcr)$, called the \textit{base measure}, which is assumed to be diffuse. In this case the distribution of $\tilde{\mu}$ will be denoted as ${\rm CRM} (\lambda; P)$, and the second integrability condition in \eqref{eq:integrability} reduces to the following
\begin{equation}
\label{eq:integrability_homogeneous}
\int_{\R^+} \min \{ s, 1\}    \lambda (s) \D s < +\infty .
\end{equation}

In the feature-sampling framework, $\tilde{\mu}$ may be used as a prior distribution if the sequence of atoms $(\tilde{\eta}_i)_{i \geq 1}$ are in between $[0,1]$, which happens if the L\'evy intensity has support on $\W  \times [0,1]$. A noteworthy example, widely used in this setting, is the stable Beta process prior (\citet{teh2009indian}). It is defined as a CRM with L\'evy intensity
\begin{equation}
\label{eq:beta_levy}
\lambda (s) =  \alpha  \cdot \frac{\Gamma (1+c)}{\Gamma (1-\sigma) \Gamma (c+\sigma)} s^{-1-\sigma}  (1-s)^{c+\sigma-1} 
 \mathds{1}_{(0,1)} (s)
\end{equation}
where $c >0$, $\sigma \in (0,1)$ and $\alpha >0$ (\citet{james2017bayesian,masoero2019more}). Now, we describe the predictive distribution an arbitrary CRM $\tilde{\mu}$. For the sake of clarity, we fix the following notation 
\[
{\rm Poiss} (y; C) := \frac{C^y e^{-C}}{y! } , y \in \N \text{ and }  {\rm Bern} (a; p) := p^a (1-p)^{1-a},  a \in \{0,1\} 
\]
to denote the probability mass functions of a Poisson with parameter $C >0$ and a Bernoulli random variable with parameter $p \in [0,1]$, respectively. We refer to \citet{james2017bayesian} for a detailed posterior analysis of CRM priors. See also \citet{broderick2018posteriors} and references therein.

\begin{Theorem}[\citet{james2017bayesian}]
\label{thm:predictive_CRM}
Let  $Z_1, Z_2, \ldots $ be exchangeable random variables modeled as in \eqref{eq:feature_ex}, where $\Mcr$ equals ${\rm CRM} (\lambda; P)$. If $\bm{Z}_n$ is a random sample which displays $K_{n}=k$ distinct features $\{W_{1}^{*},\ldots,W^{*}_{K_{n}}\}$, and feature $W_i^*$ appears exactly $M_{n,i}=m_i$ times in the samples, as $i =1, \ldots , K_n$, then
\begin{equation}
\label{eq:pred_CRM}
\begin{split}
&\P ((Y_{n+1}, A_{n+1,1}^*, \ldots , A_{n+1,K_n}^*)= (y, a_1, \ldots , a_{K_n}) | \bm{Z}_n) \\
& \qquad \qquad =  {\rm Poiss} \left( y; \int_0^1  s (1-s)^n  \lambda (s ) \D s \right) \prod_{i=1}^k {\rm Bern} (a_i; p_i^*)
\end{split}
\end{equation}
being
\[
p_i^* : = \frac{\int_0^1 s^{m_i+1} (1-s)^{n-m_i}  \lambda (s) \D s }{\int_0^1 s^{m_i} (1-s)^{n-m_i}  \lambda (s) \D s}.
\]
\end{Theorem}
\begin{proof}
We consider \citet[Proposition 3.2]{james2017bayesian} for Bernoulli product models (see also \citet[Proposition 1]{camerlenghi2021scaled}), thus the distribution of $Z_{n+1}$, given $\bm{Z}_n$, equals the distribution of
\begin{equation}
\label{eq:posterior_representation_CRM}
Z_{n+1}'+  \sum_{i=1}^{K_n} A_{n+1,i}^*  \delta_{W_i^*},
\end{equation}
where $Z_{n+1}' | \tilde{\mu}' = \sum_{i \geq 1} A_{n+1,i}' \delta_{\tilde{w}_i'} \sim {\rm BeP} (\tilde{\mu}')$ such that $\tilde{\mu}' \sim {\rm CRM} ((1-s)^n \lambda; P)$, and  $A_{n+1,1}^*, \ldots , A_{n+1,K_n}^*$ are Bernoulli random variables with parameters $J_1, \ldots , J_{K_n}$, respectively, such that each $J_i$ is a random variable whose  distribution with density function of the form
\[
f_{J_i} (s) \propto  (1-s)^{n-m_i} s^{m_i} \lambda (s).
\]
By exploiting the previous predictive characterization, we can derive the posterior distribution of $Y_{n+1}$, given $\bm{Z}_n$, by means of a direct application of the Laplace functional. Indeed,  the distribution of $Y_{n+1}|  \bm{Z}_n$ equals $\sum_{i \geq 1} A_{n+1,i}'$. Thus, for any $t \in \R$, we have the following
\begin{align*}
\E [e^{-t Y_{n+1}}| \bm{Z}_n] & =  \E  [e^{-t  \sum_{i \geq 1} A_{n+1, i}'}]  =   \E  \Big[ \prod_{i \geq 1} e^{-t A_{n+1,i}'} \Big]
= \E  \Big[ \E   \Big[ \prod_{i \geq 1} e^{-t A_{n+1,i}'} \mid  \tilde{\mu}' \Big]  \Big]\\
& =  \E  \Big[  \prod_{i \geq 1} \Big( e^{-t} \tilde{\eta}_i'  +(1- \tilde{\eta}_i') \Big)  \Big],
\end{align*}
where we used the representation $\tilde{\mu}' = \sum_{i \geq 1} \tilde{\eta}_i' \delta_{\tilde{w}_i'}$ and the fact that the  $A_{n+1,i}'$s are independent Bernoulli random variables conditionally on $\tilde{\mu}'$. We now use the Laplace functional for $\tilde{\mu}'$ to get 
\begin{align*}
\E [e^{-t Y_{n+1}}| \bm{Z}_n] & =  \E  \left[ \exp \left\{ \sum_{i \geq 1} \log (1+\tilde{\eta}_i'  (e^{-t}-1)) \right\} \right] \\
& =  \exp \left\{ -(1-e^{-t}) \int_0^1   (1-s)^n s \lambda (s)  \D s  \right\}.
\end{align*}
As a direct consequence, the posterior distribution of $Y_{n+1}$ given  $\bm{Z}_n$ is a Poisson distribution with mean $\int_0^1   (1-s)^n s \lambda (s)  \D s $. Again, by exploiting the predictive representation \eqref{eq:posterior_representation_CRM}, the posterior distribution of   $A_{n+1,i}^*$, as $i = 1, \ldots , K_n$, is a Bernoulli with the following mean
\[
\E [J_i] =  \int_0^1  s f_{J_i} (s) \D s =  \frac{\int_0^1   (1-s)^{n-m_i} s^{m_i+1} \lambda (s)  \D s   }{\int_0^1  
(1-s)^{n-m_i} s^{m_i} \lambda (s)  \D s } .
\]
\end{proof}

\begin{Corollary}
\label{cor:pred_beta}
Let $Z_1, Z_2, \ldots $ be exchangeable random variables modeled as in \eqref{eq:feature_ex}, where $\Mcr$ is the law of the stable Beta process. If $\bm{Z}_n$ is a random sample which displays $K_{n}=k$ distinct features $\{W_{1}^{*},\ldots,W^{*}_{K_{n}}\}$, and feature $W_i^*$ appears exactly $M_{n,i}=m_i$ times in the samples, as $i =1, \ldots , K_n$, then
\begin{equation}
\label{eq:pred_beta}
\begin{split}
&\P ((Y_{n+1}, A_{n+1,1}^*, \ldots , A_{n+1,K_n}^*)= (y, a_1, \ldots , a_{K_n}) | \bm{Z}_n) \\
& \qquad \qquad =  {\rm Poiss} \left( y; \alpha \frac{(c+\sigma)_n}{(c+1)_n}\right) \prod_{i=1}^k {\rm Bern}  \left( a_i; 
\frac{m_i -\sigma}{n+c}\right),
\end{split}
\end{equation}
where $(x)_y= \Gamma (x+y)/\Gamma (x)$ denotes the Pochhammer symbol for $x, y >0$.
\end{Corollary}
\begin{proof}
It is sufficient to specialize Theorem \ref{thm:predictive_CRM} for the stable Beta process. In particular, from Theorem \ref{thm:predictive_CRM} the posterior distribution of $Y_{n+1}$ given $\bm{Z}_n$ is a Poisson distribution with mean
\begin{align*}
 \int_0^1  s (1-s)^n  \lambda (s ) \D s \stackrel{\eqref{eq:beta_levy}}{=}  \frac{\alpha \Gamma (1+c)}{\Gamma (1-\sigma) \Gamma (c+\sigma)} \int_{0}^1  s^{-\sigma}(1-s)^{n+c+\sigma} \D s  = \alpha  \frac{(c+\sigma)_n}{(c+1)_n} . 
\end{align*}
Moreover the parameters of the Bernoulli random variables $A_{n+1,1}^*, \ldots , A_{n+1,K_n}^*$ are equal to
\begin{align*}
p_i^*  = \frac{\int_0^1 s^{m_i+1} (1-s)^{n-m_i}  \lambda (s) \D s }{\int_0^1 s^{m_i} (1-s)^{n-m_i}  \lambda (s) \D s}
\stackrel{\eqref{eq:beta_levy}}{=} \frac{B (m_i+1-\sigma,  c+\sigma+n -m_i)}{B (m_i -\sigma, c+\sigma+n-m_i)}
= \frac{m_i -\sigma}{n+c}
\end{align*}
as $i=1, \ldots , K_n$.
\end{proof}

\subsection{SP priors}

From Theorem \ref{thm:predictive_CRM}, under CRM priors the distribution of the number of new features $Y_{n+1}$ is a Poisson distribution which depends on the sampling information only through the sample size $n$. Moreover, the probability of observing a feature already observed in the sample, say $W_i^*$, depends only on the sample size $n$ and the frequency $m_i$ of feature $W_i^*$ out of the initial sample. \citet{camerlenghi2021scaled} showed that SP priors allow to enrich the predictive structure of CRM priors, including additional sampling information in the probability of discovering new features. To introduce SP priors, consider a CRM $\tilde{\mu}=  \sum_{i \geq 1}
\tilde{\tau}_i \delta_{\tilde{w}_i}$ on $\W$, where $(\tilde{\tau}_i)_{i \geq 1}$ are positive random atoms and $(\tilde{w}_i)_{i \geq 1}$
are i.i.d. random atoms, with L\'evy intensity $ \nu ( \D w, \D s  ) = \lambda (s) \D s P (\D w)$ satisfying
\begin{equation}
\label{eq:levy_cond}
\int_0^\infty  \min \{ s,1 \} \lambda (s) \D s < +\infty .
\end{equation}
Consider the ordered jumps $\Delta_1 > \Delta_2 > \cdots$ of the CRM $\tilde{\mu}$ and define the random measure
\[
\tilde{\mu}_{\Delta_1} =  \sum_{i \geq 1} \frac{\Delta_{i+1}}{\Delta_1} \delta_{\tilde{w}_i}
\]
normalizing $\tilde{\mu}$ by the largest jump. The definition of SPs follows by a suitable change of measure of $\Delta_1$ (\citet{james2015scaled,camerlenghi2021scaled}). Let us denote by $ \Lcr (\, \cdot \, , a  )$ a regular version of the conditional probability distribution of $(\Delta_{i+1}/\Delta_1)_{i \geq 1}$, given $\Delta_1= a$. Now denote by $\Psi_1$  a positive random variable with density function $f_{\Psi_1}$ on $\R^+$, and define
\[
\Lcr (\, \cdot \,) := \int_{\R^+}  \Lcr (\, \cdot \, ,  a  ) f_{\Psi_1} (a) \D a 
\]
the distribution of $(\Delta_{i+1}/\Delta_1)_{i \geq 1}$ obtained by mixing $ \Lcr (\, \cdot \, , a  )$ with respect to the density function $f_{\Psi_1}$. Thus, we are ready to define a SP.
\begin{Definition}
A Scaled Process (SP) prior  on $(\W, \Wcr)$ is defined as the almost surely discrete random measure
\begin{equation}
\label{eq:SP_def}
\tilde{\mu}_{\Psi_1} := \sum_{i \geq 1} \tilde{\eta}_i \delta_{\tilde{w}_i},
\end{equation}
where $(\tilde{\eta}_i)_{i \geq 1}$ has distribution $\Lcr$ and $(\tilde{w}_i)_{i \geq 1}$ is a sequence of independent random variables with common distribution $P$, also independent of $(\tilde{\eta}_i)_{i \geq 1}$. We will write
$\tilde{\mu}_{\Psi_1} \sim {\rm SP} (\nu,f_{\Psi_1} )$.
\end{Definition}
 A thoughtful account with a complete posterior analysis for SPs is given in \citet{camerlenghi2021scaled}. Here we characterize the predictive distribution \eqref{eq:predictive_general} of SPs.
\begin{Theorem} 
\label{thm:predictive_SP}[\citet{james2017bayesian,camerlenghi2021scaled}]
Let $Z_1, Z_2, \ldots $ be exchangeable random variables modeled as in \eqref{eq:feature_ex}, where $\Mcr$ equals ${\rm SP} (\nu, f_{\Psi_1})$. If $\bm{Z}_n$ is a random sample which displays $K_{n}=k$ distinct features $\{W_{1}^{*},\ldots,W^{*}_{K_{n}}\}$, and feature $W_i^*$ appears exactly $M_{n,i}=m_i$ times in the samples, as $i =1, \ldots , K_n$, then the
conditional distribution of $\Psi_1$, given $\bm{Z}_n$, has posterior density:
\begin{equation}
\label{eq:posterior_jump}
f_{\Psi_1| \bm{Z}_n} (a) \propto e^{-\sum_{i=1}^n \int_0^1  s (1-s)^{n-1} a \lambda (as) \D s } 
\prod_{i=1}^k \int_0^1  s^{m_i} (1-s)^{n-m_i}  a \lambda (as) \D s f_{\Psi_1} (a).
\end{equation}
Moreover, conditionally on $\bm{Z}_n$ and $\Psi_1$,
\begin{equation}
\label{eq:pred_SP}
\begin{split}
&\P ((Y_{n+1}, A_{n+1,1}^*, \ldots , A_{n+1,K_n}^*)= (y, a_1, \ldots , a_{K_n}) | \bm{Z}_n, \Psi_1) \\
& \qquad \qquad =  {\rm Poiss} \left( y; \int_0^1  s \Psi_1 (1-s)^n  \lambda (s \Psi_1 ) \D s \right) \prod_{i=1}^k {\rm Bern} (a_i; p_i^*(\Psi_1))
\end{split}
\end{equation}
being
\[
p_i^*(\Psi_1) : = \frac{\int_0^1 s^{m_i+1} (1-s)^{n-m_i}  \lambda (s \Psi_1) \D s }{\int_0^1 s^{m_i} (1-s)^{n-m_i}  \lambda (s \Psi_1) \D s}.
\]
\end{Theorem}
\begin{proof}
The representation of the predictive distribution \eqref{eq:pred_SP} follows from \citet[Proposition 2]{camerlenghi2021scaled}. Indeed the posterior distribution of the largest jump directly follows from \cite[Equation (4)]{camerlenghi2021scaled}. In addition \cite[Proposition 2]{camerlenghi2021scaled} shows that the conditional distribution of $Z_{n+1}$, given $\bm{Z}_n$ and $\Psi_1$, equals the distribution of the following counting measure 
\begin{equation}
\label{eq:posterior_representation_SP}
Z_{n+1}'+  \sum_{i=1}^{K_n} A_{n+1,i}^*  \delta_{W_i^*},
\end{equation}
where $Z_{n+1}' | \tilde{\mu}' = \sum_{i \geq 1} A_{n+1,i}' \delta_{\tilde{w}_i'} \sim {\rm BeP} (\tilde{\mu}_{\Psi_1}')$ and $\tilde{\mu}_{\Psi_1}'$ is a CRM with  L\'evy intensity of the form
\[
\nu_{\Psi_1}' (\D w, \D s ) =  (1-s)^n  \Psi_1 \lambda (\Psi_1 s) \indic_{(0,1)} (s)  \D s P (\D w).
\] 
Moreover $A_{n+1,1}^*, \ldots , A_{n+1,K_n}^*$ are Bernoulli random variables with parameters $J_1, \ldots , J_{K_n}$, respectively, such that conditionally on $\Psi_1$ each $J_i$ has distribution with density function of the form
\[
f_{J_i|\Psi_1} (s) \propto  (1-s)^{n-m_i} s^{m_i}  \Psi_1 \lambda (\Psi_1 s)  \quad  \text{on } (0,1).
\]
As in the proof of Theorem \ref{thm:predictive_CRM}, we show that the distribution of $Y_{n+1}| ( \Psi_1, \bm{Z}_n)$
equals $\sum_{i \geq 1} A_{n+1,i}'$. Thus, by the evaluation of the Laplace functional, one may easily realize that the last random sum has a  
Poisson distribution with mean $\int_0^1   (1-s)^n s  \Psi_1 \lambda (\Psi_1 s)   \D s $.  Moreover by exploiting the posterior representation \eqref{eq:posterior_representation_SP}, the variables   $A_{n+1,i}^*$, as $i = 1, \ldots , K_n$, conditionally on 
$\bm{Z_n}$ and $\Psi$, are independent and  Bernoulli distributed with mean 
\[
\E [J_i | \Psi_1] =  \int_0^1  s f_{J_i| \Psi_1} (s) \D s =  \frac{\int_0^1   (1-s)^{n-m_i} s^{m_i+1} \Psi_1\lambda (s \Psi_1)  \D s   }{\int_0^1  
(1-s)^{n-m_i} s^{m_i} \Psi_1\lambda (s \Psi_1)  \D s } .
\]
\end{proof}
\begin{Remark}
According to \eqref{eq:posterior_jump}, the conditional distribution of $\Psi_1$, given $\bm{Z}_{n}$ may include the whole sampling information, depending on the specification of $\nu$ and $f_{\Psi_1}$, and hence the conditional distribution of $Y_{n+1}$, given $\bm{Z}_n$, may also include such a sampling information. As a corollary of Theorem \ref{thm:predictive_SP}, the conditional distribution of $Y_{n+1}$, given $\bm{Z}_n$, is a mixture of Poisson distributions that may include the whole sampling information; in particular, the amount of sampling information in the posterior distribution is uniquely determined by the mixing distribution, namely by the conditional distribution of $\Psi_1$, given $\bm{Z}_{n}$.
\end{Remark}
Hereafter, we specialize Theorem \ref{thm:predictive_SP} for the stable SP, that is a peculiar SP defined through a CRM with a L\'evy intensity $\nu$ such that $\lambda (s) =  \sigma s^{-1-\sigma}$ for a parameter $\sigma \in (0,1)$. We refer to \citet{camerlenghi2021scaled} for a detailed posterior analysis of the stable SP prior.

\begin{Corollary}
\label{cor:predictive_SSP}
Let $Z_1, Z_2, \ldots $ be exchangeable random variables modeled as in \eqref{eq:feature_ex}, where $\Mcr$ equals ${\rm SP} (\nu, f_{\Psi_1})$, with $\lambda(s) = \sigma s^{-1-\sigma}$ for some $\sigma \in (0,1)$. If $\bm{Z}_n$ is a random sample which displays $K_{n}=k$ distinct features $\{W_{1}^{*},\ldots,W^{*}_{K_{n}}\}$, and feature $W_i^*$ appears exactly $M_{n,i}=m_i$ times in the samples, as $i =1, \ldots , K_n$, then the
conditional distribution of $\Psi_1$, given $\bm{Z}_n$, has posterior density:
\begin{equation}
\label{eq:posterior_jump_stable}
f_{\Psi_1| \bm{Z}_n} (a) \propto a^{-k \sigma} e^{-\sigma a^{-\sigma} \sum_{i = 1}^n B(1-\sigma, i)}  f_{\Psi_1} (a)
\end{equation}
having denoted by $B(\, \cdot\,  , \, \cdot \, )$ the classical Euler Beta function. Moreover, conditionally on $\bm{Z}_n$ and $\Psi_1$,
\begin{equation}
\label{eq:pred_stable}
\begin{split}
&\P ((Y_{n+1}, A_{n+1,1}^*, \ldots , A_{n+1,K_n}^*)= (y, a_1, \ldots , a_{K_n}) | \bm{Z}_n, \Psi_1) \\
& \qquad \qquad =  {\rm Poiss} \left( y; \sigma \Psi_1^{-\sigma} B (1-\sigma, n+1) \right) \prod_{i=1}^k {\rm Bern} \left( a_i; \frac{m_i-\sigma}{n-\sigma+1}\right) .
\end{split}
\end{equation}
\end{Corollary}
\begin{proof}
The proof is a plain application of Theorem \ref{thm:predictive_SP} under the choice $\lambda (s) = \sigma s^{-1-\sigma}$.
\end{proof}

%%%%%%%%%%%%%%%%%%%%%%%%%%%%%%%%
%%%%%%%%%%%%%%%%%%%%%%%%%%%%%%%%
%%%%%%%%%%%%%%%%%%%%%%%%%%%%%%%%
%%%%%%%%%%%%%%%%%%%%%%%%%%%%%%%%

\section{Predictive characterizations for SPs}  \label{sec:SP_predictive}

In this section, we introduce and discuss Johnson's ``sufficientness'' postulates in the context of features-sampling models, under the class of SP priors. According to Theorem \ref{thm:predictive_CRM}, if the features-sampling model is a CRM prior, then the conditional distribution of $Y_{n+1}$, given $\bm{Z}_n$, is a Poisson distribution that depends on the sampling information $\bm{Z}_n$ only through the sample size $n$. Moreover, the conditional probability of generating an old feature $W_{i}^*$, given $\bm{Z}_n$, depends on the sampling information $\bm{Z}_n$ only through $n$ and $m_i$. As shown in Theorem \ref{thm:predictive_SP}, SP priors enrich the predictive structure of CRM priors through the conditional distribution of the latent variable $\Psi_1$, given the observable sample $\bm{Z}_n$. In the next theorem, we characterize the class of SP priors for which the conditional distribution of $Y_{n+1}$, given $\bm{Z}_n$, depends on the sampling information only through $n$. 
\begin{Theorem} \label{thm:characterization_n}
Let $Z_1, Z_2, \ldots $ be exchangeable random variables modeled as in \eqref{eq:feature_ex}, where $\Mcr$ equals ${\rm SP} (\nu, f_{\Psi_1})$, being $\nu (\D w, \D s) = \lambda (\D s) \D s P (\D w)$. Moreover, suppose that  $\bm{Z}_n$ is a random sample which displays $K_{n}=k$ distinct features $\{W_{1}^{*},\ldots,W^{*}_{K_{n}}\}$, and feature $W_i^*$ appears exactly $M_{n,i}=m_i$ times in the samples, as $i =1, \ldots , K_n$. If $f_{\Psi_1}: (0, r) \to \R^+$ is a continuous function on the compact support $(0,r)$, with $r >0$, and the function $\lambda : \R^+ \to \R^+$ is continuous on its domain, then the conditional distribution of the latent variable $\Psi_1$ given $\bm{Z}_n$, depends on the sampling information $\bm{Z}_n$ only through $n$ if and only if $\lambda (s) = C s^{-1}$ on $(0,r)$ for some constant $C>0$.
\end{Theorem}
\begin{proof}
First of all, if $f_{\Psi_1}$ is defined on the compact support $(0,r)$ and if
 $\lambda (s) = C s^{-1}$ on $(0,r)$ for some constant $C>0$, then  it is easy to see that the posterior distribution of $\Psi_1$ in \eqref{eq:posterior_jump} depends only on $n$ and not on the other sample statistics. We now show the reverse implication. 
The posterior density of $\Psi_1$, conditionally on $\bm{Z}_n$, satisfies  \eqref{eq:posterior_jump} and it is proportional to
\begin{equation*}
f_{\Psi_1| \bm{Z}_n} (a) \propto \prod_{i=1}^n  e^{-\phi_i (a)}   \prod_{i=1}^{K_n} \int_0^1    s^{m_i} (1-s)^{n-m_i} a \lambda (a s) \D s \: 
f_{\Psi_1} (a),
\end{equation*}
where $\phi_i (a) = \int_0^{1} s (1- s)^{i-1}  a \lambda(a s)\D s$. Then, there exists $c(m_1, \ldots , m_k, k,n)$ such that it holds
\begin{equation} \label{eq:char1_1}
f_{\Psi_1| \bm{Z}_n} (a) = \frac{\prod_{i=1}^n  e^{-\phi_i (a)}   \prod_{i=1}^{K_n} \int_0^1    s^{m_i} (1-s)^{n-m_i} a \lambda (a s) \D s \: 
f_{\Psi_1} (a)}{c(m_1, \ldots , m_k, k,n)}.
\end{equation}
Because of the assumptions imposed, the distribution of $\Psi_1| \bm{Z}_n$ does not depend on $K_n$ neither on the corresponding sample frequencies $M_{n,1}, \ldots , M_{n,K_n}$. Accordingly, the function
\begin{equation}
\label{eq:char1_2} 
f_1 (a,n):=  f_{\Psi_1| \bm{Z}_n}^{-1} (a)  \prod_{i=1}^n e^{-\phi_i (a)} , \quad  a \in (0,r),
\end{equation}
depends only on $a$ and $n$ but not on $k$ and $(m_1, \ldots , m_k)$. Then, putting together \eqref{eq:char1_1}-\eqref{eq:char1_2}, it holds
\begin{equation}
\label{eq:char1_3}
f_1 (a,n) \cdot \prod_{i=1}^k  \int_0^1  s^{m_i} (1-s)^{n-m_i} a \lambda (as) \D s = c (m_1, \ldots , m_k, n,k) \quad \forall a \in (0,r),
\end{equation}
where $c$ is the normalizing factor and it does not depend on the variable $a$. By choosing $m_1= \ldots = m_k = n \in \N$, thanks to  Equation 
\eqref{eq:char1_3}, we can state that the following function
\begin{equation}
\label{eq:char1_4}
f_1 (a,n) \left( \int_0^1  s^n  a \lambda (as) \D s  \right)^k,
\end{equation}
which is defined for any $a \in (0,r)$, does not depend on $a$, but only on $k$ and $n$. Since the previous assertion is true for any $k \geq 1$, one may select $k=1$, thus 
obtaining the following identity
\begin{equation}
\label{eq:f1}
f_1 (a,n) = c^* \left( \int_0^1  s^n  a \lambda (as) \D s  \right)^{-1}
\end{equation}
for some constant $c^*$, independent of $a$, but that may depend on $n$. Substituting \eqref{eq:f1} in \eqref{eq:char1_4}, we obtain that
\begin{equation}
\label{eq:char1_5}
 c^* \left( \int_0^1  s^n  a \lambda (as) \D s  \right)^{k-1} 
\end{equation}
is a function which does not depend on $a$, but only on $n$ and $k$. As a consequence, we have that 
\[
\int_0^1  s^n  a \lambda (as) \D s =  \int_0^a  \frac{s^n}{a^n} \lambda (s) \D s = C^{**}
\]
for a suitable constant $C^{**}$, which does not depend on $a \in (0,r)$. To conclude, we take a derivative of the previous expression with respect to $a$, and this allows to show that
\[
a^n \lambda (a) =  n a^{n-1} C^{**},
\]
namely $\lambda (a) = C /a $, for $a \in (0,r)$, where $C$ is a positive constant. This is a L\'evy intensity, indeed it satisfies the condition \eqref{eq:integrability_homogeneous}. Outside the interval $(0,r)$, $\lambda$ may be defined arbitrarily, indeed the values of $\lambda$ on $[ r +\infty)$ do not affect the posterior distribution of $\Psi_1$ \eqref{eq:posterior_jump}.
\end{proof}
\begin{Remark}
Note that in Theorem \ref{thm:characterization_n} we have supposed that $f_{\Psi_1}$ has a compact support on $(0,r)$, thus we are interested to define $\lambda$ on $(0,r)$; outside the interval $\lambda$ can be defined arbitrarily,  because it does not affect the posterior distribution \eqref{eq:posterior_jump} of $\Psi_1$.
From the proof of Theorem \ref{thm:characterization_n}, it becomes apparent that if the support of $f_{\Psi_1}$ is the entire positive real line $\R^+$,  the posterior distribution of the largest jump depends only on $n$ if and only if $\lambda (s) = C s^{-1}$ on $\R^+$ for some constant $C>0$. However, in this case $\lambda$ does not meet the integrability  condition \eqref{eq:integrability_homogeneous}, hence this can only considered a limiting case. It is interesting to observe that such a limiting situation, with the additional assumption $f_{\Psi_1}= f_{\Delta_1}$, corresponds to the Beta process case with $\sigma =0$ and $c=1$ (\citet{Gri_11}).
\end{Remark}
 Now, we characterize SPs for which the posterior distribution of $\Psi_1$ depends only on $n$ and $K_n$, but not on the sample frequencies of the different features $\bm{m}$. Here we assume that $f_{\Psi_1}$ has full support a priori. The following characterization has been provided in \citet[Theorem 3]{camerlenghi2021scaled}, but for completeness we report the proof.

\begin{Theorem}\label{camerlenghi2021scaled}[\citet{camerlenghi2021scaled}]
Let $Z_1, Z_2, \ldots $ be exchangeable random variables modeled as in \eqref{eq:feature_ex}, where $\Mcr$ equals ${\rm SP} (\nu, f_{\Psi_1})$, being $\nu (\D w, \D s) = \lambda (\D s) \D s P (\D w)$. Suppose that  $\bm{Z}_n$ is a random sample which displays $K_{n}=k$ distinct features $\{W_{1}^{*},\ldots,W^{*}_{K_{n}}\}$, and feature $W_i^*$ appears exactly $M_{n,i}=m_i$ times in the sample, as $i =1, \ldots , K_n$. If $f_{\Psi_1}: \R^+ \to \R^+$ is a strictly positive function on $\R^+$ and continuously differentiable, and $\lambda$ is continuously differentiable, then the conditional distribution of the latent variable $\Psi_1$, given $\bm{Z}_n$, depends on $\bm{Z}_n$ only through $n$ and $K_n$ if and only if $\lambda (s) = C s^{-1-\sigma}$ on $\R^+$ for some constant $C>0$ and $\sigma \in (0,1)$.
\end{Theorem}
\begin{proof}
By arguing as in the proof of Theorem \ref{thm:characterization_n}, the posterior density of $\Psi_1$, given $\bm{Z}_n$, is proportional to
\begin{equation*}
\prod_{i=1}^n  e^{-\phi_i (a)}   \prod_{i=1}^{k} \int_0^1    s^{m_i} (1-s)^{n-m_i} a \lambda (a s) \D s \: 
f_{\Psi_1} (a),
\end{equation*}
where  $\phi_i (a) = \int_0^{1} s (1- s)^{i-1}  a \lambda(a s)\D s$. Then, there exists $c (m_1, \ldots , m_k, n,k)$ such that it holds
\begin{equation*}
f_{\Psi_1 | \bm{Z}_n} (a) = \frac{\prod_{i=1}^n  e^{-\phi_i (a)}  \prod_{i=1}^{k} \int_0^1    s^{m_i} (1-s)^{n-m_i} a \lambda (a s) 
\D s \: 
f_{\Psi_1}(a)}{c (m_1, \ldots , m_k, n,k)}.
\end{equation*}
As a consequence, 
\begin{equation} \label{eq:char2_1}
f_{\Psi_1 | \bm{Z}_n}^{-1} (a) \prod_{i=1}^n  e^{-\phi_i (a)}  \prod_{i=1}^{k} \int_0^1    s^{m_i} (1-s)^{n-m_i} a \lambda (a s)
 \D s\: f_{\Psi_1} (a)= c (m_1, \ldots , m_k, n,k).
\end{equation}
If the density function $f_{\Psi_1 | \bm{Z}_n} (a)$ does not depend on $m_1, \ldots, m_k$, then the following function
\[
    f_{\Psi_1 | \bm{Z}_n}^{-1} (a) \prod_{i=1}^n  e^{-\phi_i (a)} \: f_{\Psi_1}(a) = f_1(a,k,n)
\]
depends only on $k,n$ and $a$, but not  on the frequency counts. Therefore, \eqref{eq:char2_1} boils down to
\begin{equation} \label{eq:thm_char_2}
f_1 (a,k,n) \cdot  \prod_{i=1}^{k} \int_0^1  s^{m_i} (1-s)^{n-m_i} a \lambda (a s) \D s =  c (m_1, \ldots , m_k, n,k).
\end{equation}
where the function on the right hand side of \eqref{eq:thm_char_2} is independent of $a$, for any choice of the vector of sampling information $(m_1, \ldots , m_k, n,k)$. Now, since the vector $(m_1, \ldots , m_k, n,k)$ can be chosen arbitrarily, we can make the choice $m_1 = \cdots = m_k=m >0$, such that the function
\begin{equation} \label{eq:thm:eq_w}
\left[ w(a,k,n) \int_0^1    s^{m} (1-s)^{n-m} a \lambda (a s) \D s \right]^{k}
\end{equation}
does not depend on $a\in \R^+$, where $w(a,k,n) = \sqrt[k]{f_1(a,k,n)}$. Moreover, suppose that $m= n$, thus
\begin{equation} \label{eq:thm_char_3}
w(a,k,n) \int_0^1 s^n a \lambda (as) \D s 
\end{equation}
does not depend on $a\in \R^+$, which implies that
\begin{equation}
\label{eq:w}
w(a,k,n)= c^* \left( \int_0^1  s^n a \lambda (as)  \D s  \right)^{-1}
\end{equation}
for a constant $c^* >0$ with respect to $a$, which can only depend on $k$ and $n$. By substituting \eqref{eq:w} in \eqref{eq:thm:eq_w}, we obtain
\begin{equation*} 
    \left[ \frac{c^*}{\int_0^1 s^n \lambda  (a s) \D s} \cdot \int_0^1    s^{m} (1-s)^{n-m}  \lambda (a s) \D s  \right]^{k},
\end{equation*}
which is independent of $a \in \R^+$. Now, it is possible to choose $m=n-1$ in the previous function. Therefore there exists a constant $c^{**}$ independent of $a$ such that the following identity holds
\[
\int_0^1 s^{n-1} \lambda (as) \D s   -\int_0^1 s^n \lambda (as) \D s  = c^{**}  \int_0^1  s^n \lambda (as  )\D s .
\]
By taking the derivative of the previous equation two times with respect to $a$, then one obtains
\[
\lambda (a) (1-n c^{**})= a \lambda' (a) c^{**} ,
\]
which is an ordinary differential equation in $\lambda$ that can be solved by separation of variables. In particular, we obtain
\begin{equation}
    \lambda (a) = C  a^{(1-nc^{**} )/c^{**}}, \quad\text{for } C >0. \label{eq:lambda_general}
\end{equation}
To conclude, observe that the exponent of $a$  in \eqref{eq:lambda_general} should satisfy the integrability  condition \eqref{eq:integrability_homogeneous} for homogeneous CRMs. Accordingly, it is easy to see that we must consider 
\[
\lambda (a) = C \frac{1}{a^{1+\sigma}}
\]
where $C >0$ and $\sigma \in (0,1)$. The reverse implication of the theorem is trivially satisfied, hence the proof is completed.
\end{proof}

We recall from Theorem \ref{thm:predictive_SP} that the conditional distribution of $\Psi_1$, given $\bm{Z}_{n}$ uniquely determines the amount of sampling information included in the conditional distribution of the number of new features $Y_{n+1}$, given $\bm{Z}_n$. Such a sampling information may range from the whole information, in terms of $n$, $K_{n}$, and $(M_{1,n},\ldots,M_{K_{n},n})$, to the sole information on the sample size $n$. According to Theorem \ref{camerlenghi2021scaled}, stable SP prior of Corollary \ref{cor:predictive_SSP} is the sole SP prior for which the conditional distribution of the number of new features $Y_{n+1}$, given $\bm{Z}_n$, depends on the sampling information $\bm{Z}_n$ only on $n$ and $K_n$. Moreover, according to Theorem \ref{thm:characterization_n}, the Beta process prior is the sole SP prior for which the conditional distribution of the number of new features $Y_{n+1}$, given $\bm{Z}_n$, depends on the sampling information $\bm{Z}_n$ only on $n$. In particular, Theorem \ref{thm:characterization_n} and Theorem \ref{camerlenghi2021scaled} show that the Beta process prior and the stable SP prior may be considered, to some extent, the feature sampling counterparts of the Dirichlet process prior the Pitman-Yor process prior.

%%%%%%%%%%%%%%%%%%%%%%%%%%%%%%%%
%%%%%%%%%%%%%%%%%%%%%%%%%%%%%%%%
%%%%%%%%%%%%%%%%%%%%%%%%%%%%%%%%
%%%%%%%%%%%%%%%%%%%%%%%%%%%%%%%%

\section{Discussion and conclusions}\label{discuss}

In this paper, we have introduced and discussed Johnson's ``sufficientness'' postulates in the context of features-sampling models. ``Sufficientness'' postulates have been investigated extensively in the context of species-sampling models, providing an effective classification of species-sampling models on the basis of the form of their corresponding predictive distributions. Here, we made a first step towards the problem of providing an analogous classification for features-sampling models. In particular, we obtained  Johnson's ``sufficientness'' postulates when the class of features-sampling models is restricted to the class of scaled process priors. However, the results presented in the paper remain preliminary, and do not provide at all a complete answer to the characterization problem within the general class of features-sapling models. This problem remains open.

Within the features-sampling setting, the predictive distribution is of the form \eqref{eq:predictive_general}, though for the purpose of providing ``sufficientness'' postulates one may focus on features-sampling models exhibiting a general predictive distribution of the following type
\begin{equation}
\label{eq:predictive_discussion}
\begin{split}
&\P  ((Y_{n+1}, A_{n+1,1}^*, \ldots , A_{n+1,K_n}^*) = (y, a_1, \ldots , a_{K_n}) | \bm{Z}_n)\\
& \qquad\qquad\qquad\qquad\qquad\qquad  =
g (y; n,k,\bm{m}) \prod_{i=1}^{k} f_i (a_i; n,k, \bm{m}).
\end{split}
\end{equation}
Note that \eqref{eq:predictive_discussion} is a probability distribution, and it must satisfy a consistency condition, as usual. Among all the features-sampling models whose predictive distribution can be written in the form \eqref{eq:predictive_discussion}, we are interested in characterizing nonparametric priors such that: i) the function $g$ depends on the sampling information only through $n$, and the function $f_i$ depends only on $(n,m_i)$;
ii) $g$ depends only on $(n,k)$ and $f_i$ depends only on $(n,m_i)$; iii) $g$ depends only on $(n,k)$ and $f_i$ depends only on $(n,k,m_i)$.
In our view, these characterizations may provide a complete picture of sufficientness postulates within the feature setting, and they are also fundamental to guide the selection of the prior distribution. We conjecture that CRMs are the nonparametric priors satisfying the characterization i), the SP with a stable L\'evy measure is an example of prior satisfying ii), and no examples satisfying iii) have been considered in the current literature. Results in this direction are in \citet{Battiston2018}, where the authors characterize exchangeable feature allocation probability function (\citet{broderick2013feature}) in product forms: this could be a stimulating point of departure to study the characterization problem
depicted above.

\section*{Acknowledgements}

%FC is extremely grateful to Professor Eugenio Regazzini for the time spent at the Department of Mathematics of University of Pavia during his Ph.D. studies in Mathematical Statistics; FC wants to especially thank Professor Eugenio Regazzini for having introduced him to the study of Bayesian Statistics with a stimulating Ph.D. course held together with Professor Antonio Lijoi. SF wish to express his gratitude to Professor Eugenio Regazzini, whose fundamental contributions to Bayesian statistics have always been a great source of inspiration, transmitting enthusiasm and method for the development of his own research.
This research received funding from the European Research Council (ERC) under the European Union's Horizon 2020 research and innovation programme under grant agreement No 817257. The authors gratefully acknowledge the financial support from the Italian Ministry of Education, University and Research (MIUR), ``Dipartimenti di Eccellenza" grant 2018-2022. FC is a member of the \textit{Gruppo Nazionale per l'Analisi Matematica, la Probabilit\`a e le loro Applicazioni} (GNAMPA) of the \textit{Istituto Nazionale di Alta Matematica} (INdAM).

\bibliographystyle{abbrvnat}
%\bibliography{./../jabref-new}
\bibliography{bibliography}

\end{document}